\documentclass[11pt]{amsart}
\usepackage{graphicx}
\usepackage{color, fullpage, amsmath, amsthm, amssymb,verbatim,enumerate}

\newtheorem{theorem}{Theorem}[section]

\newtheorem{corollary}[theorem]{Corollary}
\newtheorem{proposition}[theorem]{Proposition}
\newtheorem{definition}[theorem]{Definition}
\numberwithin{equation}{section}

\theoremstyle{definition}

\theoremstyle{definition}

\theoremstyle{definition}

\newcommand{\C}{\mathbb{C}}
\newcommand{\D}{\mathbb{D}}

\begin{document}

\date{\today}
\title{On Cubic Blaschke Products}

\author{Alastair N. Fletcher}
\email{afletcher@niu.edu}
\address{Department of Mathematical Sciences, Northern Illinois University, DeKalb, IL 60115-2888, USA}

\author{Alexandra Hill}
\email{ahill@elgin.edu}
\address{Elgin Community College, 1700 Spartan Drive, Elgin, IL 60123}

\begin{abstract}
We show that every cubic Blaschke product has a unique hyperbolic inflection point in the unit disk and, moreover, this point lies at the hyperbolic midpoint of the two critical points. Using this structure result for cubic Blaschke products, we give an explicit expression in terms of the parameters which determines when cubic Blaschke products are elliptic, parabolic, or hyperbolic. 
\end{abstract}

\maketitle

\section{Introduction}

\subsection{Background}

Complex dynamics is the study of the iteration of holomorphic or meromorphic functions. This is a very well-developed theory, see for example \cite{Bea91,CarGam93,Mil06}. The Fatou set is the set of stable behaviour of the iterates, while the Julia set is the set of chaotic behaviour.

The proper holomorphic self-maps of $\C$ are none other than polynomials. In this setting, the orbits of critical points play a key role. If all of the orbits of the critical points remain bounded under iteration, then the Julia set is connected, whereas if all of the orbits of the critical points are unbounded, then the Julia set is a Cantor set. 

If we consider unicritical polynomials, that is, those polynomials with one critical point, then the dynamical behaviour is encoded by what happens to this critical point. Every unicritical polynomial is linearly conjugate to one of the form $z^d+c$. As conjugation does not change anything essential about the dynamics, we are thus led to the Multibrot set which, for $d\geq 2$, is defined by
\[ \mathcal{M}_d = \left \{ c \in \C : J(z^d+c) \text{ is connected}\right \}.\]
When $d=2$, there is only one critical point and so the dynamics of quadratic polynomials is encoded by the Mandelbrot set. However, once $d\geq 3$, we can have different critical orbits that escape and remain bounded. Then classifying dynamical behaviour is more complicated.

In this paper, we will investigate this circle of ideas for proper holomorphic self-maps of the disk $\D$. Such maps are finite Blaschke products, that is, maps of the form
\begin{equation} 
\label{eq:blaschke}
B(z) = \mu \prod _{i=1}^d \left ( \frac{ z-w_i}{1-\overline{w_i}z} \right ),
\end{equation}
for $d\geq 1$, $|\mu| = 1$, $w_i \in \D$ for $i=1,\ldots, d$ and $z\in \D$. The Blaschke product is called non-trivial if $d\geq 2$ (and otherwise it is a M\"obius automorphism of $\D$).

It is fairly immediate that the Julia set of a Blaschke product must be contained in $\partial \D$. It turns out that there are two possibilities: either $J(B) = \partial \D$ or it is a Cantor subset of $\partial \D$.

The reason that the dynamical possibilities are simpler for Blaschke products is the Denjoy-Wolff Theorem. As a consequence of this result, every orbit in $\D$ of a non-trivial Blaschke product must converge to a point $w_0 \in \overline{\D}$. This special point is called the Denjoy-Wolff point of $B$ and leads to a classification of Blaschke products. The Blaschke product $B$ is called:
\begin{enumerate}[(a)]
    \item {\it elliptic} if $w_0 \in \D$,
    \item {\it parabolic} if $w_0 \in \partial \D$ and $B'(w_0) = 1$,
    \item {\it hyperbolic} if $w_0 \in \partial \D$ and $0<B'(w_0)<1$.
\end{enumerate}

The dynamics of unicritical Blaschke products were studied in \cite{Fle15,CFY17}. There it was shown that if $B$ is a unicritical Blaschke product, then it is M\"obius conjugate to a unique Blaschke product of the form
\[ B_{w,d}(z) = \left ( \frac{z-w}{1-\overline{w}z} \right )^d. \]
The parameter space for degree $d$ Blaschke products can therefore be identified with $\D$.
The analogous set to the Mandelbrot set in this setting is the connectedness locus
\[ \mathcal{C}_d  = \left \{ w \in \D : J(B_{w,d}) \text{ is connected} \right \}. \]
It is well-known that the Julia set of a Blaschke product is connected if and only if $B$ is either elliptic or parabolic with $B''(w_0)=1$. It was shown in \cite{CFY17, Fle15} that the parabolic parameters for a degree $d$ Blaschke product form an epicycloid with $d-1$ cusps. The elliptic parameters lie inside this epicycloid and the hyperbolic parameters lie outside. See Figure \ref{fig:nephroid} for the degree $3$ case. Then $\mathcal{C}_d$ is the bounded component of the complement of the epicycloid together with the cusps of the epicycloid. In particular, $\mathcal{C}_d$ is neither open nor closed.

\begin{figure}
    \centering
    \includegraphics[width=3in]{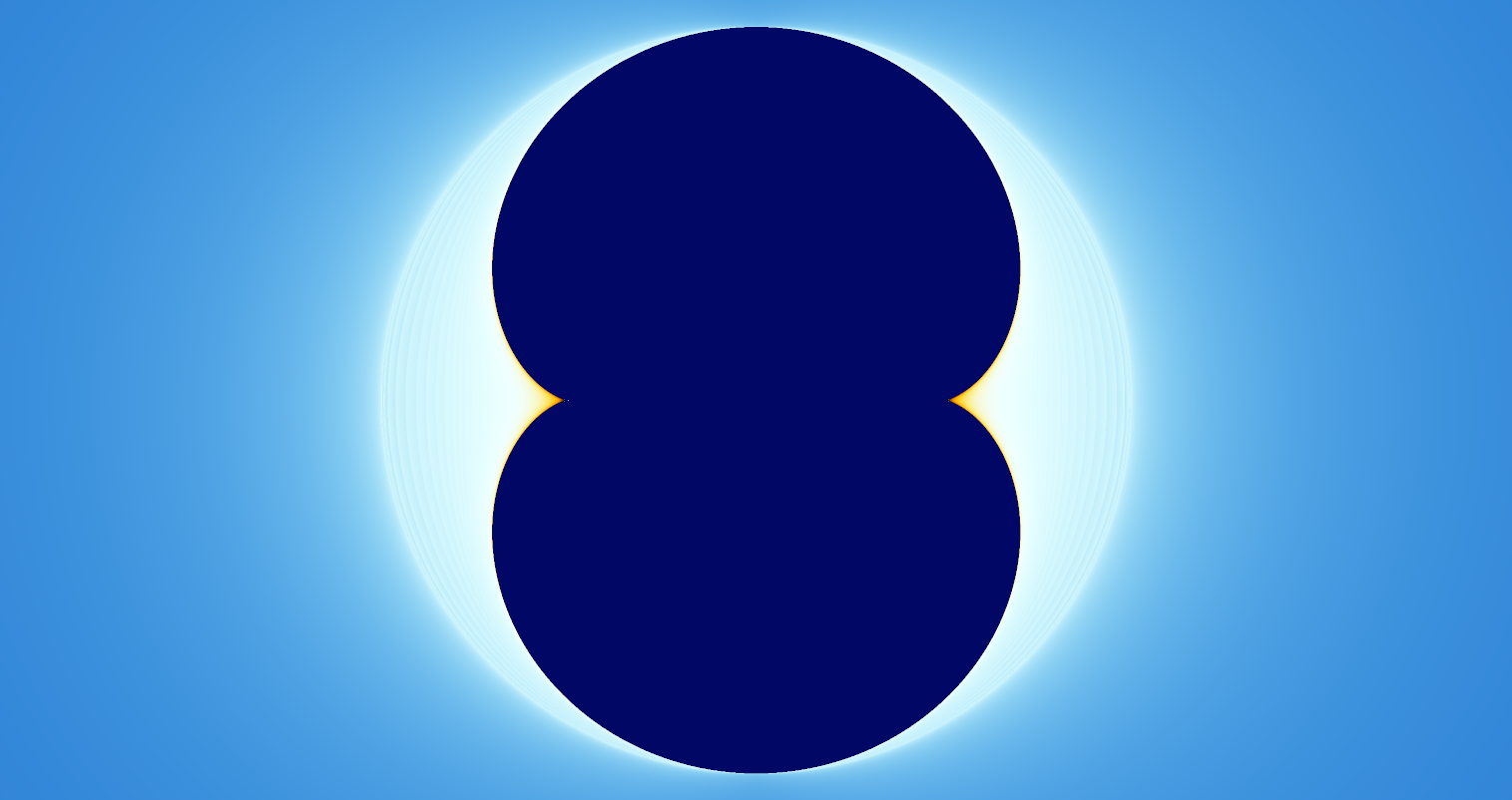}
    \caption{The parameter space for unicritical degree $3$ Blaschke products. The nephroid curve corresponds to the parabolic parameters.}
    \label{fig:nephroid}
\end{figure}

Comparing with the Multibrot set, the Blaschke situation is simpler. The Multibrot set contains an epicycloid whose interior corresponds to the case where the polynomial $z^d+c$ has an attracting fixed point. See Figure \ref{fig:multibrot} for the degree $3$ case. This feature persists to the Blaschke product setting, but all of the other complicated fractal behaviour goes away, essentially because a Blaschke product cannot have periodic attracting cycles that are not fixed.

\begin{figure}
    \centering
    \includegraphics[width=3in]{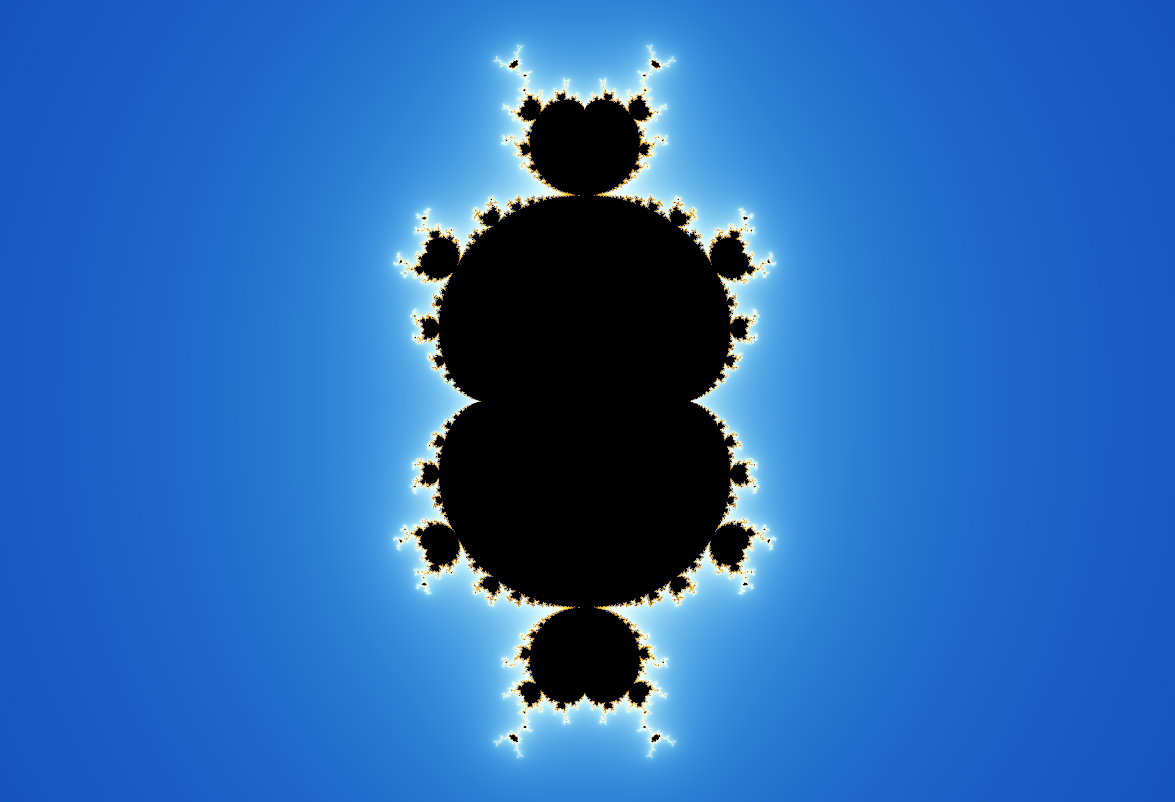}
    \caption{The parameter space for unicritical degree $3$ polynomials. Observe the nephroid appears here too.}
    \label{fig:multibrot}
\end{figure}

\subsection{Statement of Results}

The aim of this paper is to establish a dynamical classification for all degree $3$ Blaschke products. From now on, we call such Blaschke products cubic. First, we establish a normal form for cubic Blaschke products that corresponds to the form $z^3+az+b$ for cubic polynomials. Denote by $\mathcal{A}(\D)$ the set of a M\"obius automorphisms of $\D$.

\begin{proposition}
    \label{prop:normal form}
\begin{enumerate}[(a)]
\item Every cubic Blaschke product is conjugate by an element of $\mathcal{A}(\D)$ to a Blaschke product of the form
\[ B(z) = \mu \left ( \frac{ z^3 - s_0^2r_0 z^2 - s_0^2z + r_0}{\overline{r_0}z^3 - s_0^2z^2 - s_0^2 \overline{r_0}z  + 1} \right ), \]
where $s_0 \in [0,1)$, $r_0 \in \D$, and $|\mu| = 1$.
\item Every cubic Blaschke product is conjugate by an element of $\mathcal{A}(\D)$ to a Blaschke product of the form
\begin{equation}
    \label{eq:normal form}
B(z) = \frac{ z^3 - srz^2 - \overline{s}z +r }{\overline{r}z^3 - sz^2 - \overline{sr}z +1} , 
\end{equation}
where $r,s \in \D$. The form of \eqref{eq:normal form} is unique up to identifying $r$ with $-r$.
\end{enumerate}
\end{proposition}

We have stated Proposition \ref{prop:normal form} in this way as part~(a) may be found in an unpublished note of Singer \cite{Sin} that we wish to give due credit to. Part~(b) follows quickly from part~(a) and shows that we may identify parameter space for cubic Blaschke products with $\D^2$.

Our first main result is a type of structure theorem for cubic Blaschke products. The motivation for this result arises from \cite[Theorem 3.1]{CFY17} where it was shown for a degree $2$ Blaschke product that the unique critical point lies at the hyperbolic midpoint of the two zeros of $B$. One may hope that for cubic Blaschke products there may be a similar result involving solutions of $B''(z) = 0$, as there are precisely two critical points. However, it is far from clear how many solutions of this equation a cubic Blaschke product can have in the disk. Further issues arise due to the fact that the set of Blaschke products is not closed under differentiation, and that higher derivatives of Blaschke products may not behave in a conformally invariant way.

Instead of using derivatives, we will use hyperbolic derivatives $H^nB$ as defined by Rivard \cite{Riv11}. We will recall these in more detail in Section \ref{ss:hdd}, but note here that critical points of a Blaschke product $B$ are precisely the solutions of $(H^1B)(z) = 0$, and we will call solutions of $(H^2B)(z) = 0$ {\it hyperbolic inflection points}. We will see that hyperbolic inflection points will typically not be solutions of $B''(z)=0$.

\begin{theorem}
    \label{thm:structure}
Let $B$ be a cubic Blaschke product. Then $B$ has precisely one hyperbolic inflection point in $\D$ and, moreover, it lies at the hyperbolic midpoint of the two critical points of $B$.
\end{theorem}

The point about the normal forms from Proposition \ref{prop:normal form} is then revealed by the facts that in part~(a), the unique hyperbolic inflection point lies at the origin and the two critical points lie at $\pm \epsilon$ on the real axis for some $\epsilon \in [0,1)$. In part~(b), the hyperbolic inflection point still lies at the origin, but the trade off of conjugating away the $\mu$ term is that the critical points now lie at $\pm \epsilon e^{i\theta}$.

Turning to dynamical matters, as every degree $2$ Blaschke product is unicritical, a complete classification for their dynamics was obtained in \cite[Theorem 3.4]{CFY17}. As not every cubic Blaschke product is unicritical, the work of \cite{CFY17} only covers a special case in the cubic setting. Our second main result is the following classification for cubic Blaschke products.

\begin{theorem}
    \label{thm:dynamics}
Suppose that $B$ is of the form \eqref{eq:normal form}. Then $B$ is hyperbolic, parabolic, or elliptic if and only if $P(r,s)$ is positive, zero, or negative respectively, where
\begin{align*}
P(r,s) &= \left(\left( 16|r|^2-\left|1+s\right|^2\right)^2-\left|8\overline{r}\left(\overline{sr}-sr\right)+3\left(1+s\right)^2\right|^2\right)^2 \\
&\hskip0.2in -\left|\left(\left|1+s\right|^{2}-16\left|r\right|^{2}\right)\left(2\left(1+s\right)\left(sr-\overline{\left(sr\right)}\right)+12\overline{r}\left(1+\overline{s}\right)\right) \right . \\
&\hskip0.5in \left . +\left(3\left(1+s\right)^{2}+8\overline{r}\left(\overline{\left(sr\right)}-sr\right)\right)\left(2\left(1+\overline{s}\right)\left(\overline{sr}-sr\right)+12r\left(1+s\right)\right)\right|^{2}.
\end{align*}
\end{theorem}

While the formula for $P(r,s)$ may not be the nicest, it does provide a complete classification for cubic Blaschke products in terms of the parameters. The strategy in \cite{CFY17} was to identify the parabolic parameters by simultaneously solving $B(z)=z$ and $B'(z)=1$. This is less amenable to solving in the cubic setting, and so we adopt a different approach. We will instead find the hyperbolic parameters via an application of the Schur-Cohn Algorithm. 

The paper is organized as follows. In Section \ref{s:prelims}, we will cover the background material we need on the hyperbolic metric, Blaschke products and their properties, hyperbolic divided differences, and the Schur-Cohn algorithm and related matters. In Section \ref{s:inflection} we will prove Proposition \ref{prop:normal form} and Theorem \ref{thm:structure}. In Section \ref{s:dynamics} we will apply the Schur-Cohn Algorithm to the degree two case to give another proof of results from \cite{CFY17} and then prove Theorem \ref{thm:dynamics} on the degree three case.

This paper is based on the dissertation of the second named author at Northern Illinois University, and both authors would like to thank the external examiner, Javad Mashreghi, for many useful comments. The images were generated by {\it Ultra Fractal 6}.

\section{Preliminaries}
\label{s:prelims}

We denote by $\D^*$ the complement of $\overline{\D}$ in the Riemann sphere $\C_{\infty}$.

\subsection{The hyperbolic metric}

We refer to Beardon and Minda's expository article \cite{BM07} for further details than the overview we provide here.
The unit disk $\D$ carries a complete conformal metric called the hyperbolic metric. This conformal metric has density
\[ \lambda (z) |dz| =  \frac{ 2|dz|}{1-|z|^2} \]
and distance explicitly given by
\[ \rho(z,w) = \log \left ( \frac{ 1 + \left | \frac{z-w}{1-\overline{w}{z}} \right | }{ 1 - \left | \frac{z-w}{1-\overline{w}{z}} \right | } \right ), \]
for $z,w\in \D$. Denote by $\mathcal{A}(\D)$ the set of M\"obius automorphisms of $\D$ given by
\[ A(z) = \mu \left ( \frac{z-w}{1-\overline{w}z} \right ),\]
for $|\mu| =1$ and $w\in \D$. Then a modern reformulation of the Schwarz-Pick Lemma says that the hyperbolic isometries are precisely the elements of $\mathcal{A}(\D)$ and, moreover, if $f:\D\to \D$ is a holomorphic map that is not in $\mathcal{A}(\D)$, then $f$ is a strict hyperbolic contraction. In other words, if $f:\D\to \D$ is holomorphic, then
\[ \rho(f(z),f(w)) \leq \rho(z,w),\]
for all $z,w\in \D$ with equality anywhere if and only if $f \in \mathcal{A}(\D)$.

The hyperbolic geodesics are arcs of circles which meet $\partial \D$ perpendicularly. Then the hyperbolic mid-point of $z,w\in \D$ is given by
\begin{equation}
    \label{eq:hypmidpt}
c = \frac{ |zw|^2-1 + \sqrt{ (1-|z|^2)(1-|w|^2)|1-\overline{w}z|^2 }}{\overline{zw}(z+w) - (\overline{z} + \overline{w})},
\end{equation}
see \cite[Proposition 3.2]{CFY17}.

\subsection{Blaschke Products}

We refer to the book of Garcia, Mashreghi and Ross \cite{GMR18} for more details on Blaschke products that we provide here.

Recall the definition of a finite Blaschke product from \eqref{eq:blaschke}. As a Blaschke product is a product of M\"obius factors, it is clear that the poles of $B$ lie in $\D^*$ and so $B$ is holomorphic in $\D$. As another consequence of this, $B(\D) = B^{-1}(\D) = \D$, $B(\D^*) = B^{-1}(\D^*) = \D^*$, and $B(\partial \D) = B^{-1}(\partial \D) = \partial \D$. In other words, $\D,\D^*$ and $\partial \D$ are all completely invariant sets of $B$.

We can compute that the logarithmic derivative of $B$ is
\[ \frac{B'(z)}{B(z)} = \sum_{i=1}^d \frac{ 1-|w_i|^2}{(1-\overline{w_i}z)(z-w_i)}. \]
Thus for $|z|=1$, we have
\[ |B'(z)| = \sum_{i=1}^d  \frac{ 1-|w_i|^2}{ |z-w_i|^2}.\]
As each term on the right is strictly positive, it follows that $B$ has no critical points on $\partial \D$. For us, it will be important to know how many critical points of $B$ lie in $\D$. Then \cite[Theorem 6.1.4]{GMR18} states that a degree $d$ Blaschke product has $d-1$ critical points (counting multiplicity) in $\D$. We can obtain more information on the location of critical points from the following result, see \cite[Corollary 5.1.4]{GMR18} for the Euclidean case and \cite[Theorem 6.1.6]{GMR18} for the hyperbolic case.

\begin{theorem}[Gauss-Lucas Theorem]
\label{thm:gausslucas}
\begin{enumerate}[(a)]
\item Let $p$ be a non-constant polynomial. Then the zeros of $p'$ are contained in the convex hull of the zeros of $p$.
\item Let $B$ be a finite Blaschke product. Then the zeros of $B'$ in $\D$ are contained in the hyperbolic convex hull of the zeros of $B$.
\end{enumerate}
\end{theorem}

\subsection{Complex dynamics}
\label{ss:cd}

If $f$ is a rational map, then $z\in \C_{\infty}$ is said to be in the Fatou set $F(f)$ if there exists a neighbourhood $U$ of $z$ such that $(f^n|_U)_{n=1}^{\infty}$ is a normal family. If no such neighbourhood exists, then $z$ is said to be in the Julia set.

An important characterization of the Julia set is that $z\in J(f)$ if and only if for every neighbourhood $U$ of $z$, the forward orbit or $U$,
\[ O^+(U) = \bigcup_{n=0}^{\infty} f^n(U), \]
omits at most two points in $\C_{\infty}$. It follows from this property, along with the complete invariance of $\D,\D^*$ and $\partial \D$ under a finite Blaschke product $B$ that $J(B)$ must be contained in $\partial \D$. Thus the unit disk $\D$ is always contained in the Fatou set $F(B)$ and the precise iterative behaviour there is described by the following theorem.

\begin{theorem}[Denjoy-Wolff Theorem]
\label{thm:dw}
Let $f:\D\to\D$ be holomorphic. Then either $f$ is conjugate to a rotation, or there exists $w_0 \in \overline{\D}$ such that $f^n$ converges to the constant map $w_0$ locally uniformly on $\D$.
\end{theorem}

The location of the Denjoy-Wolff point $w_0$ then leads to the classification of non-trivial finite Blaschke products as described in the Introduction. Recall that in both the elliptic and hyperbolic cases, we have $|B'(w_0)| <1$. As the location of the Denjoy-Wolff point depends continuously on the parameters in the Blaschke product, it follows that both ellipticity and hyperbolicity are open conditions in parameter space, once we suitably provide a topology on parameter space. On the other hand, the condition of parabolicity requires the Denjoy-Wolff point to satisfy $B'(w_0)=1$ (observe that when $w_0 \in \partial \D$, the value of $B'(w_0)$ is real). It follows that parabolicity is a closed condition in parameter space. Thus it is reasonable to view the parabolic parameters as the boundary between the elliptic and hyperbolic parameters in parameter space.

\subsection{Hyperbolic Divided Differences}
\label{ss:hdd}

The complex pseudo-hyperbolic distance between $z,w \in \D$ is given by
\[ [z,w] = \frac{z-w}{1-\overline{w}z}.\]
Using this, we recall the notion of hyperbolic divided differences introduced by Baribeau, Rivard and Wegert \cite{BRW09}. For simplicity, we will just apply this to finite Blaschke products

\begin{definition}
\label{def:divdiff}
Let $B$ be a finite Blaschke product and let $w\in \D$. Then the first hyperbolic divided difference of $B$ is
\[ (\Delta_w B)(z) = \frac{ [B(z),B(w)]}{[z,w]} = \frac{ \frac{ B(z)-B(w) }{1-\overline{B(w)}B(z) } }{ \frac{z-w}{1-\overline{w}z} }. \]
For $n\geq 2$, the $n$'th hyperbolic divided difference of $B$ is
\[ (\Delta _w^n B)(z) = \frac{ [ (\Delta_w^{n-1} f)(z) , (\Delta_w^{n-1} B)(w) ] }{[z,w]}.   \]
\end{definition}

One of the key points here is that the set of Blaschke products is not closed under differentiation. However taking hyperbolic divided differences does preserve the set of Blaschke products.

\begin{theorem}[Theorem 2.4, \cite{Riv11}]
\label{thm:hypdiffbl}
Let $B$ be a finite Blaschke product of degree $d$. If $0\leq k \leq d$ and $w\in \D$ then $(\Delta_w^{d-k} B)$ is a Blaschke product of degree $d-k$.
\end{theorem}

We note that \cite[Theorem 2.4]{Riv11} says somewhat more and, in particular, includes a converse that we do not need here.

\begin{definition}[Definition 3.1, \cite{Riv11}]
    \label{def:hypderiv}
Let $f$ be a finite Blaschke product and let $n\geq 1$. The hyperbolic derivative of order $n$ of $B$ at $z\in \D$, denoted by $H^nB(z)$, is defined by
\[ H^nB(z) = \Delta_z^n B(z).\]
\end{definition}

When $n=1$, we have by \cite[Theorem 3.2 (1)]{Riv11}
\begin{equation} 
\label{eq:H1}
H^1 B(z) = \frac{ B'(z) (1-|z|^2)}{1-|B(z)|^2}, 
\end{equation}
for $z\in \D$, which agrees with the hyperbolic derivative that is well-known in the literature, see for example \cite{BM07}. Moreover, when $n=2$ and $B$ is a non-trivial Blaschke product, we have by \cite[Theorem 3.2 (2)]{Riv11} that
\begin{equation} 
\label{eq:H2}
H^2B(z) = \frac{1}{2(1-|H^1B(z)|^2)} \left ( \frac{ B''(z)(1-|z|^2)^2}{1-|B(z)|^2} - \frac{ 2\overline{z} B'(z) (1-|z|^2)}{1-|B(z)|^2} + \frac{ 2\overline{B(z)} B'(z)^2 (1-|z|^2)^2 }{(1-|B(z)|^2)^2} \right ).  
\end{equation}

We make the following definition.

\begin{definition}
    \label{def:hypinflpt}
Let $B$ be a finite Blaschke product.
\begin{enumerate}[(a)]
\item A point $z\in \D$ is called a hyperbolic critical point if $H^1B(z) = 0$.
\item A point $z\in \D$ is called a hyperbolic inflection point if $H^2B(z) = 0$.
\end{enumerate}
\end{definition}

We could make the definition for higher order hyperbolic critical points in the obvious way, but we do not need this here. It follows immediately from \eqref{eq:H1} that $z$ is a hyperbolic critical point for $B$ if and only if $z$ is a critical point. On the other hand, we can see from \eqref{eq:H2} that a solution of $B''(z) = 0$ will typically not be a hyperbolic inflection point (consider $B(z) = z^3$ for a simple example where it is).

An important property of hyperbolic derivatives is their conformal invariance.

\begin{theorem}[Theorem 3.4, \cite{Riv11}]
\label{thm:confinv}
Let $B$ be a finite Blaschke product, let $\psi,\phi \in \mathcal{A}(\D)$, and let $n\geq 1$. Then
\[ |H^n (\psi \circ B \circ \phi)(z) | = |H^n(B)(\phi(z))|, \]
for $z\in \D$.    
\end{theorem}

As an immediate consequence of Theorem \ref{thm:confinv}, we have the following corollary.

\begin{corollary}
\label{cor:confinv}
Let $B$ be a finite Blaschke product and let $\psi,\phi \in \mathcal{A}(\D)$. Then $z$ is a hyperbolic critical point, or hyperbolic inflection point respectively, of $B$ if and only if $\phi^{-1}(z)$ is a hyperbolic critical point, or hyperbolic inflection point respectively, of $\psi \circ B \circ \phi$.
\end{corollary}

\subsection{The Schur-Cohn Algorithm}
\label{ss:schurcohn}

The Schur-Cohn Algorithm applies to complex polynomials and determines how many zeros lie within $\D$ or $\D^*$. We refer to \cite[Section 6.8]{Hen74} for more details than we provide here. 

In this section, we suppose that $p$ is a polynomial with complex coefficients. If
\[ p(z) = a_dz^d +a_{d-1}z^{d-1} +\ldots +a_1z + a_0,\]
where $d\geq 0$ and $a_i \in \C$ for $i=0,\ldots, d$, then the reciprocal polynomial of $p$ is defined by
\[ p^*(z) = \overline{a_0}z^d + \overline{a_1}z^{d-1} + \ldots + \overline{a_{d-1}}z + \overline{a_0}.\]
The reciprocal polynomial of a polynomial of degree $d$ is always a polynomial of degree at most $d$. In the cases we are interested in, $a_0$ will not vanish, and so the reciprocal polynomial will also be of degree $d$.

A polynomial is called self-inversive if there exists $\mu$ with $|\mu| = 1$ such that $p^*(z) = \mu p(z)$. It follows that a self-inversive polynomial has a relationship between its coefficients given by $a_k = \mu \overline{a_{n-k}}$ for $k=0,\ldots, d$.

If $p$ has zeros at $w_1,\ldots, w_d$, repeated according to multiplicity, then $p^*$ has zeros at $1/\overline{w_i}$ for $i=1,\ldots, d$, again repeated according to multiplicity.

It is worth pointing out that if $A\in \mathcal{A}(\D)$, then 
\[ A(z) = \mu \frac{p(z)}{p^*(z)},\]
for $p(z) = z-w$ and $|\mu|=1$. As a consequence, every finite Blaschke product of degree $d$ is of the form
\begin{equation} 
\label{eq:selfinv}
B(z) = \mu \frac{p(z)}{p^*(z)},
\end{equation}
where $p$ is a degree $d$ polynomial and $|\mu|=1$. Returning to our build-up for the Schur-Cohn algorithm, we recall the following definition.

\begin{definition}
    \label{def:schur}
If $p$ is a polynomial of degree $d\geq 1$, the Schur transform of $p$ is the degree $d-1$ polynomial defined by
\[ Tp(z) = \overline{p(0)}p(z) - \overline{ p^*(0)}p^*(z) = \sum_{k=0}^{d-1} \left ( \overline{a_0}a_k - a_d \overline{a_{d-k}} \right ) z^k.\]
The iterated Schur transforms are then defined for $n\geq 2$ inductively via $T^np = T(T^{n-1}p)$.
\end{definition}

The part of the Schur-Cohn Algorithm that we will use then reads as follows.

\begin{theorem}[Schur-Cohn Algorithm, Theorem 6.8b, \cite{Hen74}]
    \label{thm:schurcohn}
Let $p$ be a complex polynomial of degree $d$ and for $k\geq 1$, set $\delta_k = T^kp(0)$. Then all the roots of $p$ lie in $\D^*$ if and only if $\delta_k > 0$ for $k=1,\ldots, d$.
\end{theorem}

\section{Hyperbolic Inflection Points}
\label{s:inflection}

\subsection{Normal forms of cubic Blaschke products}

Here we establish the normal forms for our cubic Blaschke products.

\begin{proof}[Proof of Proposition \ref{prop:normal form}]
Starting with an arbitrary cubic Blaschke product $B$, by Corollary \ref{cor:confinv} we may conjugate $B$ by an element of $\mathcal{A}(\D)$ to move its two critical points in $\D$ to $\pm \epsilon$ on the real axis. Then, with a slight abuse of notation, we call the new Blaschke product $B$ and, following \cite[Section 2]{Sin} may write it in un-factored form as
\[ B(z) = -\mu \frac{ z^3 - r_2z^2 + r_1z - r_0 }{ \overline{r_0}z^3 - \overline{r_1}z^2 + \overline{r_2}z -1 }.\]
Solving $B'(z)=0$ leads to the self-inversive equation
\[ (\overline{r_0}r_2 - \overline{r_1})z^4 + (2\overline{r_2} - 2r_1\overline{r_0})z^3 + (-3-|r_2|^2+|r_1|^2 + 3|r_0|^2)z^2 + (2r_2 - 2\overline{r_1}r_0)z + (r_0\overline{r_2} - r_1) = 0. \]
As the four roots of this equation must be $\pm\epsilon$ and $\pm\epsilon^{-1}$, this means the coefficients of $z^3$ and $z$ must vanish, that is, we must have $r_2 = \overline{r_1}r_0$. Substituting this in and simplifying, we obtain
\[ \overline{r_1}z^4 + ( 3-|r_1|^2)z^2 + r_1 = 0,\]
where $|r_1|<1$. In order for the four roots of this equation to be real, $r_1$ must be real and negative. Then setting $r_1 = -s^2$ yields part~(a).

For part~(b), set $R(z) = e^{i\theta}$ and suppose $B$ is given as in part~(a). Then
\begin{align*}
(R^{-1} \circ B \circ R)(z) &= (e^{-i\theta} \mu) \left ( \frac{ e^{3i\theta} z^3 - s_0^2r_0 e^{2i\theta}z^2 - s_0^2e^{i\theta}z +r_0}{\overline{r_0}e^{3i\theta}z^3 - s_0^2e^{2i\theta}z^2 - s_0^2 \overline{r_0}e^{i\theta}z +1} \right )\\
&= \frac{ \mu e^{2i\theta}z^3 - s_0^2r_0\mu e^{i\theta}z^2 - s_0^2 \mu z + r_0 \mu e^{-i\theta}}{\overline{r_0}e^{3i\theta}z^3-s_0^2e^{2i\theta}z^2 - s_0^2 \overline{r_0}e^{i\theta} z +1}.
\end{align*}
As $|\mu| = 1$, we have $\mu = e^{i\phi}$ for some $\phi \in [0,2\pi )$. We choose $\theta = -\phi /2$ and set $r = r_0e^{-3i\theta}$ to obtain
\begin{align*} 
(R^{-1} \circ B \circ R)(z) &= \frac{ z^3 - \left( s_0e^{i\theta}\right )^2rz^2 - \left ( \overline{s_0e^{i\theta}} \right )^2z + r}{\overline{r}z^3 + \left ( s_0e^{i\theta} \right )^2 z^2 - \left ( \overline{s_0e^{i\theta}} \right )^2 \overline{r} z +1 }.  
\end{align*}
Setting $s = \left ( s_0e^{i\theta} \right )^2$ yields part~(b).

For the uniqueness claim, if $A_1,A_2$ are two elements of $\mathcal{A}(\D)$ which conjugate a cubic Blaschke product $f$ to \eqref{eq:normal form}, then $A_2\circ A_1^{-1}$ is an element of $\mathcal{A}(\D)$ which either fixes the two critical points of $B$ or switches them. In the first case, $A_2\circ A_1^{-1}$ must be the identity and in the second case $A_2\circ A_1^{-1}$ must be $-z$. As $B$ given by \eqref{eq:normal form} satisfies
\[ B_1(z) := -B(-z) = \frac{ z^3 +srz^2 - \overline{s}z-r}{-\overline{r}z^3 -sz^2 +\overline{sr}z+1}, \]
we see that $(r,s)$ and $(-r,s)$ are identified via $\mathcal{A}(\D)$ in parameter space.
\end{proof}

\subsection{Location of the hyperbolic inflection point}

In this section, we prove our structure theorem on cubic Blaschke products.

\begin{proof}[Proof of Theorem \ref{thm:structure}]
By Proposition \ref{prop:normal form}~(a), we will assume that $B$ has been conjugated to move its critical points to $\pm \epsilon$ on the real axis and thus that $B$ has the form
\[ B(z) = \mu \left ( \frac{ z^3 - s_0^2r_0 z^2 - s_0^2z + r_0}{\overline{r_0}z^3 - s_0^2z^2 - s_0^2 \overline{r_0}z  + 1} \right ), \]
for $|\mu|=1$, $s_0 \in [0,1)$, and $r_0 \in \D$. Our aim is to show that $B$ has only one hyperbolic critical point in $\D$ and that it lies at the origin. Corollary \ref{cor:confinv} then establishes the result for arbitrary cubic Blaschke products.

Rather than computing solutions of $H^2B(z) = 0$ directly by using \eqref{eq:H2}, our strategy is to use the fact, via Theorem \ref{thm:hypdiffbl}, that $\Delta_w B(z)$ is a degree two Blaschke product. By \cite[Theorem 3.1]{CFY17}, if $f$ is a degree two Blaschke product, then $\Delta_wf(w) =0$ if and only if $w$ is the hyperbolic midpoint of the roots of $f$. Thus to solve $\Delta^2_w B (w) = 0$, we need $w$ to be the hyperbolic midpoint of the roots of $\Delta_wB(z)$ when we set $z=w$.

We begin by calculating $B(z)-B(w)$ which yields

\begin{align*}
B(z) - B(w) &= \frac{\mu\left(z^{3}-s_0^{2}r_{0}z^{2}-s_0^{2}z+r_{0}\right)}{\overline{r_{0}}z^{3}-s_0^{2}z^{2}-s_0^{2}\overline{r_{0}}z+1}-\frac{\mu\left(w^{3}-s_0^{2}r_{0}w^{2}-s_0^{2}w+r_{0}\right)}{\overline{r_{0}}w^{3}-s_0^{2}w^{2}-s_0^{2}\overline{r_{0}}w+1} \\
&= \mu(\overline{r_{0}}z^{3}-s_0^{2}z^{2}-s_0^{2}\overline{r_{0}}z+1)^{-1}(\overline{r_{0}}w^{3}-s_0^{2}w^{2}-s_0^{2}\overline{r_{0}}w+1)^{-1} \\
& \hskip0.2in \cdot \left [ (z^{3}-s_0^{2}r_{0}z^{2}-s_0^{2}z+r_{0})(\overline{r_{0}}w^{3}-s_0^{2}w^{2}-s_0^{2}\overline{r_{0}}w+1) \right .\\
& \hskip0.4in \left . - (w^{3}-s_0^{2}r_{0}w^{2}-s_0^{2}w+r_{0})(\overline{r_{0}}z^{3}-s_0^{2}z^{2}-s_0^{2}\overline{r_{0}}z+1) \right ]
\end{align*}

Multiplying out the terms in the square brackets above and simplifying yields

\begin{align*}
B(z) - B(w) &= \mu(\overline{r_{0}}z^{3}-s_0^{2}z^{2}-s_0^{2}\overline{r_{0}}z+1)^{-1}(\overline{r_{0}}w^{3}-s_0^{2}w^{2}-s_0^{2}\overline{r_{0}}w+1)^{-1} \\
& \hskip0.2in \left [ s_0^{2}|r_0|^2z^{2}w^{2}\left(z-w\right)+s_0^{2}\overline{r_{0}}zw\left(z^{2}-w^{2}\right)-|r_0|^2\left(z^{3}-w^{3}\right)-s_0^{2}z^{2}w^{2}\left(z-w\right) \right .\\
& \hskip0.4in -s_0^{4}zw\left(z-w\right)+s_0^{2}r_{0}\left(z^{2}-w^{2}\right)-s_0^{2}\overline{r_{0}}zw\left(z^{2}-w^{2}\right)+s_0^{4}|r_0|^2zw\left(z-w\right) \\
& \hskip0.4in \left . + s_0^{2}|r_0|^2\left(z-w\right)+\left(z^{3}-w^{3}\right)-s_0^{2}r\left(z^{2}-w^{2}\right)-s_0^{2}\left(z-w\right) \right ]\\
&= \mu(z-w)(\overline{r_{0}}z^{3}-s_0^{2}z^{2}-s_0^{2}\overline{r_{0}}z+1)^{-1}(\overline{r_{0}}w^{3}-s_0^{2}w^{2}-s_0^{2}\overline{r_{0}}w+1)^{-1} \\
& \hskip0.2in \cdot \left [ z^{2}\left(w^{2}s_0^{2}\left(|r_0|^2-1\right)-\left(|r_0|^2-1\right)\right) \right .\\
& \hskip0.4in +z\left(-s_0^{4}w\left(1-|r_0|^2\right)+w\left(1-|r_0|^2\right)\right) \\
& \hskip0.4in \left . +\left(w^{2}\left(1-|r_0|^2\right)-s_0^{2}\left(1-|r_0|^2\right)\right)  \right ].
\end{align*}

Now, using this and simplifying again, we compute that the first hyperbolic divided difference is
\begin{align*}
\left(\Delta_{w}B\right)\left(z\right) & =\frac{B\left(z\right)-B\left(w\right)}{1-\overline{B\left(w\right)}B\left(z\right)}\cdot\frac{1-\overline{w}z}{z-w}\\
&= \frac{B(z)-B(w)}{z-w} \cdot \frac{1-\overline{w}z}{1-\overline{B(w)}B(z)}\\
&= \mu\left ( \frac{1-\overline{w}z}{1-\overline{B(w)}B(z)}\right )(\overline{r_{0}}z^{3}-s_0^{2}z^{2}-s_0^{2}\overline{r_{0}}z+1)^{-1}(\overline{r_{0}}w^{3}-s_0^{2}w^{2}-s_0^{2}\overline{r_{0}}w+1)^{-1} \\
& \hskip0.2in \cdot (|r_0|^2-1)\left [ z^2\left(w^{2}s_0^{2}-1\right)+zw\left(1-s_0^4\right)+\left(w^2-s_0^2\right)  \right ].
\end{align*}

Next, we will solve $\left(\Delta_{w}B\right)\left(z\right)=0$ to find the roots of the hyperbolic divided difference. First note that the factor $$\frac{1-\overline{w}z}{1-\overline{B\left(w\right)}B\left(z\right)}$$ will not have a zero of the hyperbolic divided difference since both $w$ and $z$ lie inside the unit disk. That leaves finding the zeros of the other factor, and namely solving 
\begin{align*}
    z^2\left( \left(w^{2}s_0^{2}-1\right)\left(|r_0|^2-1\right)\right)+z\left(w\left(1-s_0^4\right)\left(1-|r_0|^2\right)\right)+\left(\left(w^2-s_0^2\right)\left(1-|r_0|^2\right)\right) =& 0 \\
    \left(1-|r_0|^2\right)\left( z^2 \left(1-w^{2}s_0^{2}\right)+zw\left(1-s_0^4\right)+(w^2-s_0^2\right) =& 0
\end{align*}
to find the roots of $\left(\Delta_{w}B\right)\left(z\right)$. We can use the quadratic formula:

\begin{align*}
    z =& \frac{-w\left(1-s_0^{4}\right)\pm\left(\left(w\left(1-s_0^{4}\right)\right)^2-4\left(1-w^{2}s_0^{2}\right)\left(w^{2}-s_0^{2}\right)\right)^{1/2}}{2\left(w^{2}s_0^{2}-1\right)}\\
     =& \frac{-w\left(1-s_0^{4}\right)\pm\left(w^{2}\left(1-s_0^{4}\right)^{2}+4\left(w^{2}s_0^{2}-1\right)\left(w^{2}-s_0^{2}\right)\right)^{1/2}}{-2\left(w^{2}s_0^{2}-1\right)}\\
     =& \frac{-w\left(1-s_0^{4}\right)\pm\left(w^{2}s_0^{8}+4w^{4}s_0^{2}+4s_0^{2}-3w^{2}-6w^{2}s_0^{4}\right)^{1/2}}{2\left(1-w^{2}s_0^{2}\right)}.
\end{align*}

Therefore the roots of the hyperbolic divided difference are 
\[
    p=\frac{-w\left(1-s_0^{4}\right)+\left(w^{2}s_0^{2}+4w^{4}s_0^{2}+4s_0^{2}-3w^{2}-6w^{2}s_0^{4}\right)^{1/2}}{2\left(1-w^{2}s_0^{2}\right)}
\] 
and 
\[
    q=\frac{-w\left(1-s_0^{4}\right)-\left(w^{2}s_0^{2}+4w^{4}s_0^{2}+4s_0^{2}-3w^{2}-6w^{2}s_0^{4}\right)^{1/2}}{2\left(1-w^{2}s_0^{2}\right)}.
\]

To find the hyperbolic midpoint of $p$ and $q$, we will use \eqref{eq:hypmidpt}.
To start we compute that

\begin{align*}
    p+q 
     =\frac{-w\left(1-s_0^{4}\right)}{1-w^{2}s_0^{2}}
\end{align*}

and 

\begin{align*}
    pq & =\left(\frac{-w\left(1-s_0^{4}\right)+\left(w^{2}s_0^{8}+4w^{4}s_0^{2}+4s_0^{2}-3w^{2}-6w^{2}s_0^{4}\right)^{1/2}}{2\left(1-w^{2}s_0^{2}\right)}\right)\\
    &\hskip0.4in \cdot \left(\frac{-w\left(1-s_0^{4}\right)-\left(w^{2}s_0^{8}+4w^{4}s_0^{2}+4s_0^{2}-3w^{2}-6w^{2}s_0^{4}\right)^{1/2}}{2\left(1-w^{2}s_0^{2}\right)}\right)\\
    & =\frac{w^{2}\left(1-s_0^{4}\right)^{2}-\left(w^{2}s_0^{8}+4w^{4}s_0^{2}+4s_0^{2}-3w^{2}-6w^{2}s_0^{4}\right)}{4\left(1-w^{2}s_0^{2}\right)^{2}}\\
    & =\frac{4w^{2}+4w^{2}s_0^{2}-4w^{4}s_0^{2}-4s_0^{2}}{4\left(1-w^{2}s_0^{2}\right)^{2}}\\
    & =\frac{\left(w^{2}-s_0^{2}\right)\left(1-w^{2}s_0^{2}\right)}{\left(1-w^{2}s_0^{2}\right)^{2}}\\
    & =\frac{w^{2}-s_0^{2}}{1-w^{2}s_0^{2}}.
\end{align*}

Thus the hyperbolic midpoint $c$ of $p$ and $q$ satisfies the equation

\begin{align*}
  &\left[\overline{\left(\frac{w^{2}-s_0^{2}}{1-w^{2}s_0^{2}}\right)}\left(\frac{-w\left(1-s_0^{4}\right)}{1-w^{2}s_0^{2}}\right)-\overline{\left(\frac{-w\left(1-s_0^{4}\right)}{1-w^{2}s_0^{2}}\right)}\right]c^{2}\\
  &\hskip0.2in +2\left(1-\left|\frac{w^{2}-s_0^{2}}{1-w^{2}s_0^{2}}\right|^{2}\right)c+\left(\frac{w^{2}-s_0^{2}}{1-w^{2}s_0^{2}}\right)\overline{\left(\frac{-w\left(1-s_0^{4}\right)}{1-w^{2}s_0^{2}}\right)}-\left(\frac{-w\left(1-s_0^{4}\right)}{1-w^{2}s_0^{2}}\right)  =0.  
\end{align*}
Clearing the denominators, we obtain
\begin{align*}
    &\left(-w\left(1-s_0^{4}\right)\overline{\left(w^{2}-s_0^{2}\right)}+\overline{w}\overline{\left(1-s_0^{4}\right)}\left(1-w^{2}s_0^{2}\right)\right)c^{2}\\
    &\hskip0.2in +2\left(\overline{\left(1-w^{2}s_0^{2}\right)}\left(1-w^{2}s_0^{2}\right)-\left(w^{2}-s_0^{2}\right)\overline{\left(w^{2}-s_0^{2}\right)}\right)c\\
    &\hskip0.4in +\left(w^{2}-s_0^{2}\right)\overline{\left(-w\left(1-s_0^{4}\right)\right)}+w\left(1-s_0^{4}\right)\overline{\left(1-w^{2}s_0^{2}\right)}  =0.
\end{align*}

After multiplying out the coefficients of powers of $c$ and factorizing, we obtain
\begin{align*}
\left(1-s_0^{4}\right)\left [ \left(1-\left|w\right|^{2}\right)\left(\overline{w}+s_0^{2}w\right)c^{2}+2\left(1-\left|w\right|^{4}\right)c+\left(w+s_0^{2}\overline{w}\right)\left(1-|w|^2\right) \right ]=0
\end{align*}

Now we can divide by the non-zero constants $\left(1-s^{4}\right)$, as $0\leq s<1$, and $\left(1-\left|w\right|^{2}\right)$ as $\left|w\right|\neq1$. Therefore, this equation simplifies to 
\[
    \left(\overline{w}+s_0^{2}w\right)c^{2}+2\left(1+\left|w\right|^{2}\right)c+\left(w+s_0^{2}\overline{w}\right)=0.
\]

Recall that to find solutions of $(\Delta_w^2 B)(w) = 0$, we need this midpoint $c$ to agree with $w$. Thus setting $c=w$ above yields

\[
    \left(\overline{w}+s_0^{2}w\right)w^{2}+2\left(1+\left|w\right|^{2}\right)w+\left(w+s_0^{2}\overline{w}\right) =0
\]
which simplifies to 
\[    
    s_0^{2}w^{3}+3\overline{w}w^{2}+3w+s_0^{2}\overline{w}  =0.
\]

Setting $w=x+iy$ we expand and factor the polynomial to obtain

\[
    x\left(s_0^{2}x^{2}-3y^{2}s_0^{2}+3x^{2}+3y^{2}+3+s_0^{2}\right)+y\left(3x^{2}s_0^{2}-s_0^{2}y^{2}+3x^{2}+3y^{2}+3-s_0^{2}\right)i  = 0.
\]

Equating real and imaginary parts yields the system of equations:

\begin{align}
    x\left(s_0^{2}x^{2}-3y^{2}s_0^{2}+3x^{2}+3y^{2}+3+s_0^{2}\right) & =0 \label{x sols} \\ 
    y\left(3x^{2}s_0^{2}-s_0^{2}y^{2}+3x^{2}+3y^{2}+3-s_0^{2}\right) & =0. \label{y sols}
\end{align}

Note that $x=0$ and $y=0$ is a solution. Our task is to show that it is the only valid solution in $\D$. Now suppose that $x=0$ but $y\neq0$. Then \eqref{y sols} simplifies to 

\begin{align*}
    -s_0^{2}y^{2}+3y^{2}+3-s_0^{2} & =0\\
    y^{2}\left(3-s_0^{2}\right) & =-\left(3-s_0^{2}\right)\\
    y^{2} & =-1,
\end{align*}

as $3-s_0^2 \neq 0$. Thus there are no such solutions. Similarly, if $y=0$ but $x\neq0$, then \eqref{x sols} simplifies to 

\begin{align*}
    s_0^{2}x^{2}+3x^{2}+3+s_0^{2} & =0\\
    x^{2}\left(3+s_0^{2}\right) & =-\left(3+s_0^{2}\right)\\
    x^{2} & =-1,
\end{align*}
as $3+s_0^2 \neq 0$. Again, there are no such solutions. Now suppose $x\neq0$ and $y\neq0$. Then we are solving the system of equations 

\begin{align*}
    x^{2}\left(3+s_0^{2}\right)+3y^{2}\left(1-s_0^{2}\right)+\left(3+s_0^{2}\right) & =0\\
    3x^{2}\left(s_0^{2}+1\right)+y^{2}\left(3-s_0^{2}\right)+\left(3-s_0^{2}\right) & =0.
\end{align*}

Subtracting these two equations gives

\begin{align*}
    3x^{2}+x^{2}s_0^{2}-3x^{2}s_0^{2}-3x^{2}+3y^{2}-3y^{2}s_0^{2}-3y^{2}+y^{2}s_0^{2}+3+s_0^{2}-3+s_0^{2} & =0\\
    -2x^{2}s_0^{2}-2y^{2}s_0^{2}+2s_0^{2} & =0\\
    x^{2}+y^{2} & =1.
\end{align*}

Therefore, any other solutions to the original system of equations lies on the unit circle and is not valid. We conclude that the only hyperbolic inflection point of $B$ lies at the point $w=0$, which is clearly at the midpoint of the two critical points. This completes the proof.
\end{proof}

\section{Dynamics of Cubic Blaschke Products}
\label{s:dynamics}

Here, we outline the strategy we will use to classify the type of the Blaschke products under consideration. Suppose that $B$ is a hyperbolic Blaschke product. Then all fixed points of $B$ necessarily lie on $\partial \D$. Moreover, every solution to $B(z) -z = 0$ must be simple, as otherwise $B$ would be a parabolic Blaschke product.

If $B$ is a degree $d$ Blaschke product, then solving $B(z) - z = 0$ is equivalent to solving a self-inversive polynomial equation $p(z) = 0$ of degree $d+1$. Again if $B$ is hyperbolic, then all roots of $p$ are simple. By the Euclidean Gauss-Lucas Theorem, Theorem \ref{thm:gausslucas}~(a), all the critical points of $p$ lie inside the polygon formed by the roots of $p$ and none of them lie at a vertex. Thus all the critical points of $p$ are in $\D$. Then if we set $q(z) = (p')^*(z)$, recalling the notation for a reciprocal polynomial, all the roots of $q$ lie in $\D^*$ and we are in a situation that we can apply the Schur-Cohn Algorithm, Theorem \ref{thm:schurcohn}.

We will thus find a condition on the parameters of $B$ which yield the hyperbolic Blaschke products. This will necessarily be an open set in parameter space whose boundary gives the parabolic parameters and what remains will be the elliptic parameters.

\subsection{The degree two case}

To illustrate the method, we apply it to the degree two case and recover \cite[Theorem 3.4]{CFY17}.
Recalling Theorem \ref{thm:schurcohn}, we need to compute $\delta_1$ and $\delta_2$ and ascertain when they are both positive.

Given a degree two Blaschke product, we may conjugate it by an element of $\mathcal{A}(\D)$ to $B$ so that the unique critical point of $B$ lies at the origin and $B$ takes the form
\[
    B(z)=\frac{z^2-u}{1-\overline{u}z^2},
\]
for some $u\in \D$. We therefore identify the parameter space for degree two Blaschke products with $\D$.

To find fixed points, we solve $B(z)=z$ and then clear denominators to obtain the self-inversive equation
\[ p(z) := \overline{u}z^3+z^2-z-u = 0.\]

As discussed at the start of this section, if $B$ is hyperbolic then all the roots of $p'$ must lie in $\D$. We set $q$ to be the reciprocal polynomial of $p'$, that is, 
\begin{equation}  
\label{eq:q}
q(z)=(p')^*(z)=-z^2+2z+3u.
\end{equation}
Then all roots of $q$ must lie in $\D^*$. Next, we will apply the Schur-Cohn Algorithm. We observe here that the algorithm can fail if we apply it to a self-inversive polynomial as the Schur transforms degenerate. However, it is clear from \eqref{eq:q} that $q$ is not self-inversive.

Starting with the first Schur transform, we have
\begin{align*}
    Tq(z) &=\overline{q(0)}q(z)-\overline{q^*(0)}q^*(z)\\
    &= 3\overline{u}\left(-z^2+2z+3u\right) - (-1)\left(3\overline{u}z^2+2z-1\right)\\
    &= (2+6\overline{u})z+9|u|^2-1.
\end{align*}
Now for the second Schur transform, we have 
\begin{align*}
    T^2q(z) &= \overline{Tq(0)}Tq(z)-\overline{(Tq)^*(z)}(Tq)^*(z)\\
    &= \left(9|u|^2-1\right)\left((2+6\overline{u})z+9|u|^2-1\right) - \left((2+6\overline{u}\right)\left( (9|u|^2-1)z+2+6u\right) \\
    &= \left(9|u|^2-1\right)^2 - |2+6u|^2.
\end{align*}
From these Schur transforms, we can compute $\delta_1,\delta_2$ as follows. We have
\[
    \delta_1 =Tq(0)=9|u|^2-1
\]
and
\[
    \delta_2=T^2q(0)=\left(9|u|^2-1\right)^2 - |2+6u|^2.
\]

The condition that $\delta_1>0$ simplifies to $|u| > 1/3$ and so $u\in \D$ lies outside the circle centred at the origin of radius $1/3$. The condition that $\delta_2>0$ is somewhat more involved.
Setting $u=x+iy$ for some real numbers $x$ and $y$ means
\[  \left(9\left|u\right|^{2}-1\right)^{2}-\left|6\overline{u}+2\right|^{2}  >0 \]
becomes
\[ \left(9\left(x^{2}+y^{2}\right)-1\right)^{2}  > \left(6x+2\right)^{2}+\left(6y\right)^{2}. \]
By expanding out both sides and factoring in terms of $x+\tfrac13$ and $y$, we obtain
\[  \left(\left(x+\frac{1}{3}\right)^{2}+y^{2}-\frac{2}{3}\left(x+\frac{1}{3}\right)\right)^{2}  > \frac{4}{9}\left(\left(x+\frac{1}{3}\right)^{2}+y^{2}\right). \]

This inequality corresponds to the region outside the cardioid in $\D$ with cusp at $-\tfrac13$. It follows that these are the hyperbolic parameters, the parabolic parameters are those on the cardioid, and the elliptic parameters are those inside the cardioid. We therefore recover \cite[Theorem 3.4]{CFY17}.

We note that if $\delta_2>0$ then $\delta_1>0$, and the only solution of $\delta_1=\delta_2=0$ occurs at the cusp of the cardioid. This cusp is the only parabolic parameter for which $J(B)$ is all of $\partial \D$.

\subsection{The degree three case}

For the degree three case, we start by assuming that $B$ has already been conjugated into the form given by Proposition \ref{prop:normal form}~(b). Then we identify parameter space with $(r,s) \in \D \times \D$ and recall that $(r,s)$ is identified with $(-r,s)$.

We consider first two special cases. Suppose that $r=0$ so we consider the slice $\{0\} \times \D$ in parameter space. Then $B$ takes the form
\[
B\left(z\right)=\frac{z^{3}-\overline{s}z}{1-sz^{2}}.
\]
It is clear that $B(0) = 0$ and so $0$ is the Denjoy-Wolff point of $B$. Thus $B$ is always elliptic in this case.

Next suppose that $s=0$ so we consider the slice $\D \times \{ 0 \}$ in parameter space. Then $B$ takes the form
\[
B\left(z\right)=\frac{z^{3}+r}{\overline{r}z^{3}+1}.
\]
As we can write $B(z) = A(z^3)$ for $A\in \mathcal{A}(\D)$, it follows that the only critical point of the Blaschke product is when $z=0$, and therefore it is a unicritical Blaschke product. Thus we are back in the case considered in \cite{CFY17,Fle15} and the parameter space picture is as in Figure \ref{fig:nephroid}. 

Now we move to the general case.

\begin{proof}[Proof of Theorem \ref{thm:dynamics}]

Let $B$ be a hyperbolic degree three Blaschke product of the form \eqref{eq:normal form}.
To find the fixed points of $B$, we solve
\[ \frac{z^{3}-srz^{2}-\overline{s}z+r}{\overline{r}z^{3}-sz^{2}-\overline{sr}z+1}  =z, \]
which leads to the self-inversive polynomial equation
\[p\left(z\right):=\overline{r}z^{4}-\left(1+s\right)z^{3}+\left(sr-\overline{sr}\right)z^{2}+\overline{\left(1+s\right)}z-r=0. \]

As in the degree two case, all roots of $p$ are simple and contained in $\partial \D$ and by the Gauss-Lucas Theorem, all critical points of $p$ lie in $\D$. Thus if we set
\begin{align*}
q\left(z\right)  =& (p')^*\left(z\right)\\
=& (1+s)z^3+2\left( \overline{sr}-sr\right)z^2-3\overline{(1+s)}z+4r,
\end{align*}
then all roots of $q$ lie in $\D^*$. As $q$ is clearly not self-inversive, we may apply the Schur-Cohn Algorithm. The first Schur transform is 
\begin{align*}
    Tq\left(z\right)  &= \overline{q(0)}q(z)-\overline{q^*(0)}q^*(z)\\
    &= 4\overline{r}q\left(z\right)-\left(1+s\right)q^{*}\left(z\right)\\ 
    &= 4\overline{r}\left((1+s)z^3+2\left( \overline{sr}-sr\right)z^2-3\overline{(1+s)}z+4r\right) \\&\hskip0.5in - \left(1+s\right)\left(4\overline{r}z^3-3(1+s)z^2+2(sr-\overline{sr})z+\overline{(1+s)}\right)\\
    &= \left(8\overline{r}\left(\overline{sr}-sr\right)+3\left(1+s\right)^{2}\right)z^{2}-\left(12\overline{r}\overline{\left(1+s\right)}+2\left(1+s\right)\left(sr-\overline{sr}\right)\right)z\\&\hskip0.5in +16\left|r\right|^{2}-\left|1+s\right|^{2}.
\end{align*}
Now the reciprocal polynomial for $Tq$ is 
\begin{align*}
    \left(Tq\right)^{*}\left(z\right)&=\left(16\left|r\right|^{2}-\left|1+s\right|^{2}\right)z^{2}-\left(12r\left(1+s\right)+2\overline{\left(1+s\right)}\left(\overline{sr}-sr\right)\right)z\\&\hskip0.5in +8r\left(sr-\overline{sr}\right)+3\overline{\left(1+s\right)^{2}}.
\end{align*}
Using this, we can calculate the second Schur transform,
\begin{align*}
    T^{2}q\left(z\right)  &= \overline{Tq(0)}Tq(z)-\overline{(Tq)^*(z)}(Tq)^*(z)\\
    &= \left(16\left|r^{2}\right|-\left|1+s\right|^{2}\right)\left(Tq\left(z\right)\right)-\left(8\overline{r}\left(\overline{sr}-sr\right)+3\left(1+s\right)^{2}\right)\left(\left(Tq\right)^{*}\left(z\right)\right)\\
    &= \left(16\left|r^{2}\right|-\left|1+s\right|^{2}\right)\left[\left(8\overline{r}\left(\overline{sr}-sr\right)+3\left(1+s\right)^{2}\right)z^{2} \right . \\
    & \hskip0.5in \left . -\left(12\overline{r}\overline{\left(1+s\right)}+2\left(1+s\right)\left(sr-\overline{sr}\right)\right)z+16\left|r\right|^{2}-\left|1+s\right|^{2}\right]\\ &\hskip0.5in -\left(8\overline{r}\left(\overline{sr}-sr\right)+3\left(1+s\right)^{2}\right)\left[\left(16\left|r\right|^{2}-\left|1+s\right|^{2}\right)z^{2} \right . \\ &\hskip0.5in \left .-\left(12r\left(1+s\right)+2\overline{\left(1+s\right)}\left(\overline{sr}-sr\right)\right)z+8r\left(sr-\overline{sr}\right)+3\overline{\left(1+s\right)^{2}}\right]\\
     &=  \left[ \left(\left|1+s\right|^{2}-16\left|r\right|^{2}\right)\left(2\left(1+s\right)\left(sr-\overline{sr}\right)+12\overline{r}\overline{\left(1+s\right)}\right) \right. \\ & \hskip0.5in \left. +\left(8\overline{r}\left(\overline{sr}-sr\right)+3\left(1+s\right)^{2}\right)\left(12r\left(1+s\right)+2\overline{\left(1+s\right)}\left(\overline{sr}-sr\right)\right)\right]z \\ & \hskip0.5in +\left(16\left|r\right|^{2}-\left|1+s\right|^{2}\right)^{2}-\left|8\overline{r}\left(\overline{sr}-sr\right)+3\left(1+s\right)^{2}\right|^{2}.
\end{align*}

Now the reciprocal polynomial for $T^{2}q$ is
\begin{align*}
\left(T^{2}q\right)^{*}\left(z\right)&=\left(\left(16\left|r\right|^{2}-\left|1+s\right|^{2}\right)^{2}-\left|8\overline{r}\left(\overline{sr}-sr\right)+3\left(1+s\right)^{2}\right|^{2}\right)z\\ &+\overline{\left(\left|1+s\right|^{2}-16\left|r\right|^{2}\right)\left(2\left(1+s\right)\left(sr-\overline{sr}\right)+12\overline{r}\overline{\left(1+s\right)}\right)} \\ &+ \overline{\left(8\overline{r}\left(\overline{sr}-sr\right)+3\left(1+s\right)^{2}\right)\left(12r\left(1+s\right)+2\overline{\left(1+s\right)}\left(\overline{sr}-sr\right)\right)}.
\end{align*}

The last Schur transform we need to calculate is the third transform,
which is given by
\begin{align*}
T^{3}q\left(z\right) & = \overline{T^2q(0)}T^2q(z)-\overline{(T^2q)^*(z)}(T^2q)^*(z)\\
&= \left( \left(16\left|r\right|^{2}-\left|1+s\right|^{2}\right)^{2}-\left|8\overline{r}\left(\overline{sr}-sr\right)+3\left(1+s\right)^{2}\right|^{2}\right)\left(T^2q\right)(z) \\ &\quad -\left[\left(\left|1+s\right|^{2}-16\left|r\right|^{2}\right)\left(2\left(1+s\right)\left(sr-\overline{sr}\right)+12\overline{r}\overline{\left(1+s\right)}\right) \right . \\ & \left . \quad +\left(8\overline{r}\left(\overline{sr}-sr\right)+3\left(1+s\right)^{2}\right)  \right . \\ & \quad \left . \cdot\left(12r\left(1+s\right)+2\overline{\left(1+s\right)}\left(\overline{sr}-sr\right)\right)\right]\left(T^{2}q\right)^{*}\left(z\right)\\
 &= \left(\left( 16|r|^2-\left|1+s\right|^2\right)^2-\left|8\overline{r}\left(\overline{sr}-sr\right)+3\left(1+s\right)^2\right|^2\right)^2 \\ &\quad -\left|\left(\left|1+s\right|^{2}-16\left|r\right|^{2}\right)\left(2\left(1+s\right)\left(sr-\overline{sr}\right)+12\overline{r}\overline{\left(1+s\right)}\right) \right . \\ & \quad \left . +\left(3\left(1+s\right)^{2}+8\overline{r}\left(\overline{sr}-sr\right)\right)\left(2\overline{\left(1+s\right)}\left(\overline{sr}-sr\right)+12r\left(1+s\right)\right)\right|^{2}.
\end{align*}

Now that we have all the Schur transforms, we can find $\delta_{1}, \delta_2$, and $\delta_3$:
\begin{align*}
    \delta_{1} =& Tq\left(0\right)=16\left|r\right|^{2}-\left|1+s\right|^{2},\\
    \delta_{2} =& T^{2}q\left(0\right)=\left(16\left|r\right|^{2}-\left|1+s\right|^{2}\right)^{2}-\left|3\left(1+s\right)^{2}+8\overline{r}\left(\overline{sr}-sr\right)\right|^{2},\\
    \delta_{3} =& T^{3}q\left(0\right)=\left(\left( 16|r|^2-\left|1+s\right|^2\right)^2-\left|8\overline{r}\left(\overline{sr}-sr\right)+3\left(1+s\right)^2\right|^2\right)^2 \\ 
    &\hskip0.6in -\left|\left(\left|1+s\right|^{2}-16\left|r\right|^{2}\right)\left(2\left(1+s\right)\left(sr-\overline{sr}\right)+12\overline{r}\overline{\left(1+s\right)}\right) \right . \\ 
    & \hskip0.6in \left . +\left(3\left(1+s\right)^{2}+8\overline{r}\left(\overline{sr}-sr\right)\right)\left(2\overline{\left(1+s\right)}\left(\overline{sr}-sr\right)+12r\left(1+s\right)\right)\right|^{2}.
\end{align*}

We see from these equations that $\delta_1$ is the first squared term in $\delta_2$, and $\delta_2$ is the first squared term in $\delta_3$. Therefore if we require $\delta_1,\delta_2,\delta_3$ to all be strictly positive, it's enough to check that $\delta_3>0$.

Therefore, we obtain the hyperbolic parameters $(r,s)$ by applying Theorem \ref{thm:schurcohn} and noting that $P(r,s)$ is precisely the formula given by $\delta_3$ above. The boundary of this region in $\D^2$ will be the parabolic parameters and thus corresponds to the equation $P(r,s) = 0$. Finally, the only case that remains is that of the elliptic parameters and so that corresponds to $P(r,s) <0$. This completes the proof.
\end{proof}

\subsection{Discussion and Images}

We include several images of slices in parameter space for fixed values of $s$. Recall that $s=0$ corresponds to the nephroid in Figure \ref{fig:nephroid}. As we vary the $s$-parameter along the real axis, the nephroid shape persists. As $s$ increases along the real axis, the shape gets wider and $s$ decreases, the shape gets narrower, as illustrated by Figures \ref{fig:cubic1} and \ref{fig:cubic2}.

\begin{figure}
    \centering
    \includegraphics[width=3in]{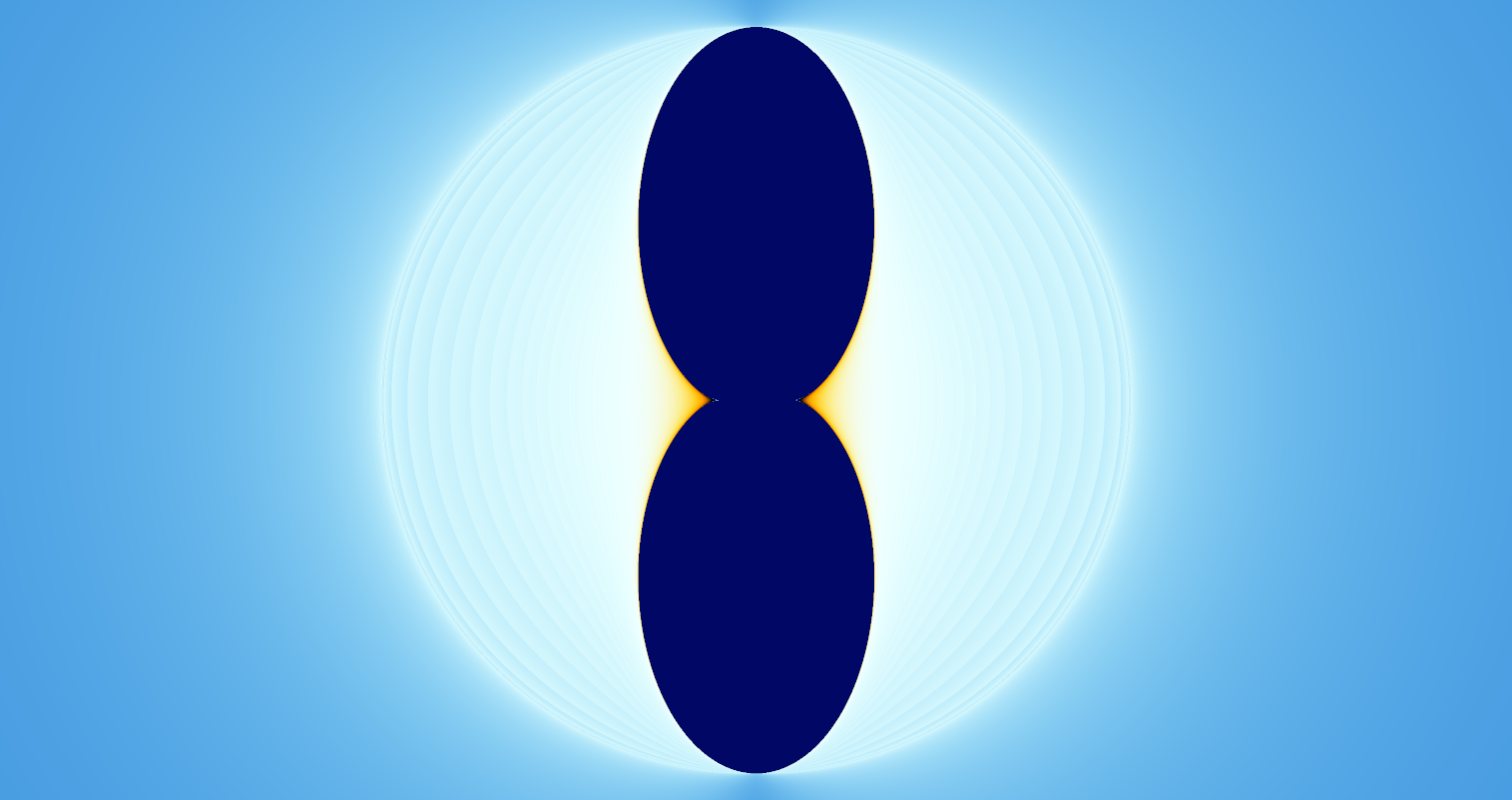}
    \caption{The slice in parameter space given by $s=-0.8$.}
    \label{fig:cubic1}
\end{figure}

As $s$ varies along the imaginary axis, the nephroid shape remains, but it distorts as the cusps move in different directions, see Figure \ref{fig:cubic3}. As $s$ heads towards $\partial \D$, the shape distorts towards a circle with two arms. Note that the condition that $\delta_1>0$ corresponds to
\[ |r| > \frac{|1+s|}{4}\]
and so the set of elliptic parameters in a given $s$-slice must always contain a disk centred at the origin of radius at least $\tfrac14$.

\begin{figure}
    \centering
    \includegraphics[width=3in]{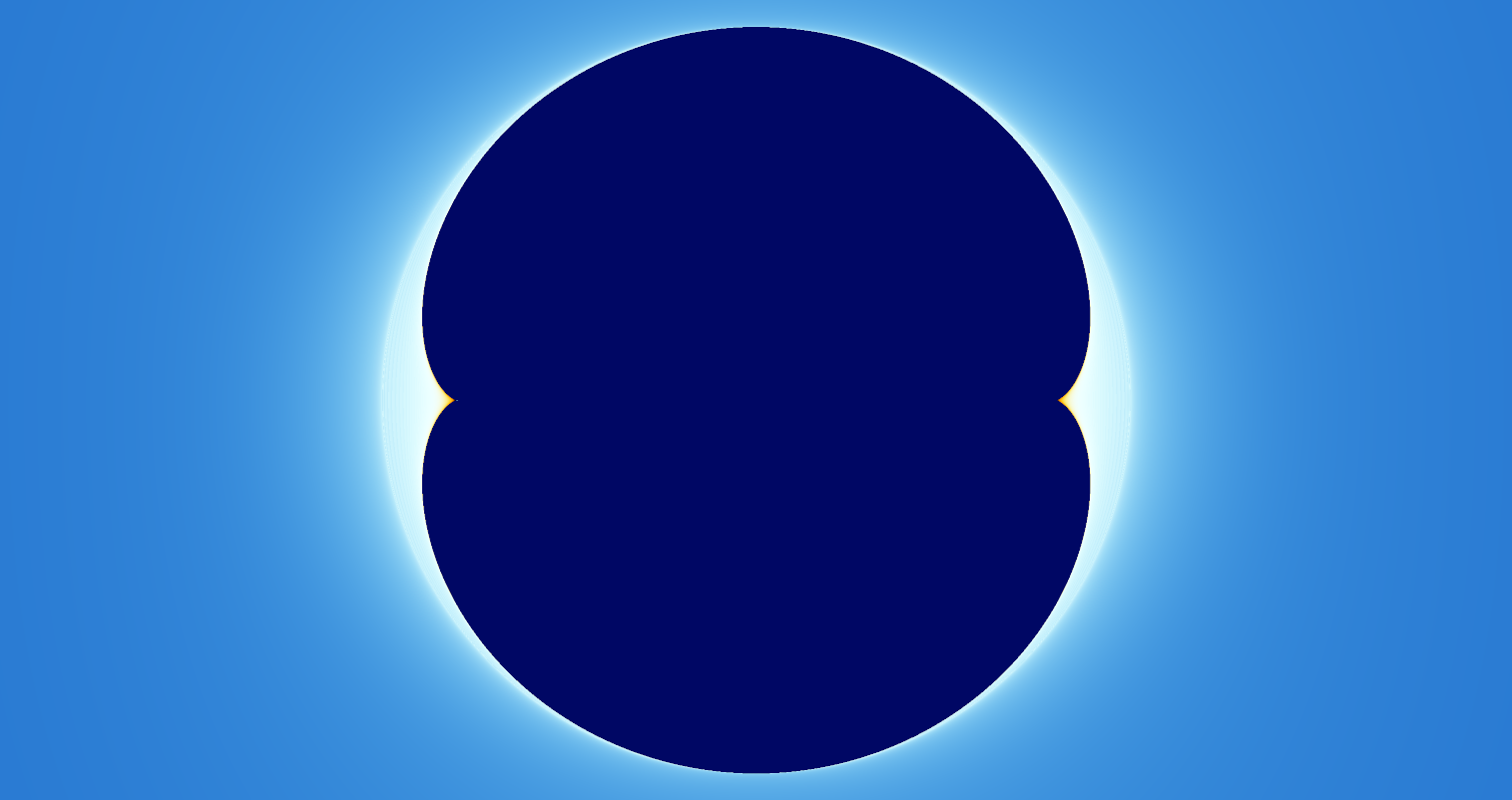}
    \caption{The slice in parameter space given by $s=0.6$.}
    \label{fig:cubic2}
\end{figure}

We also observe that all of these figures exhibit the rotational invariance that arises from the fact that $r$ and $-r$ are identified in parameter space through conjugating the Blaschke product by $-z$.

\begin{figure}
    \centering
    \includegraphics[width=3in]{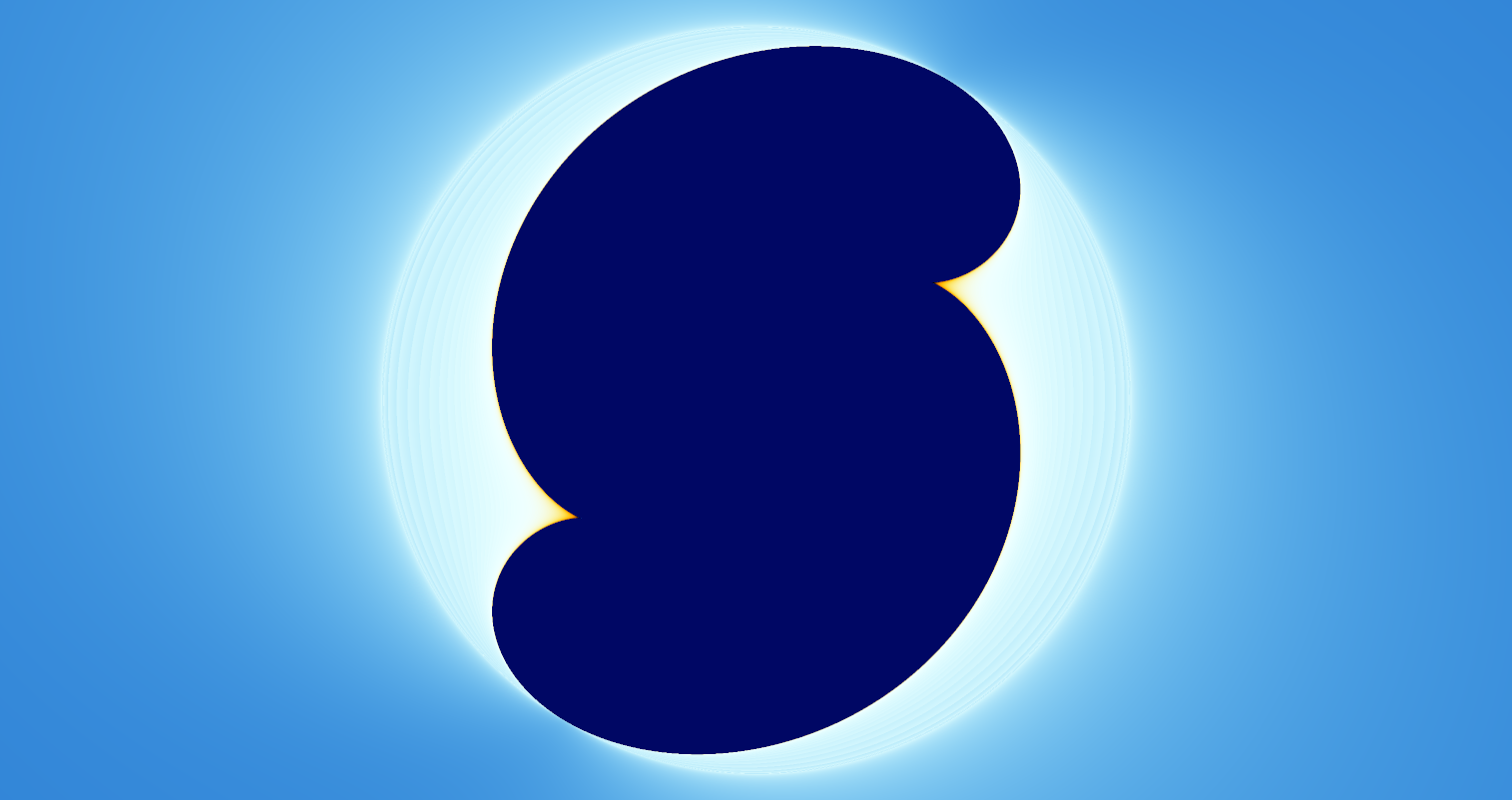}
    \caption{The slice in parameter space given by $s=-0.6i$.}
    \label{fig:cubic3}
\end{figure}

\begin{figure}
    \centering
    \includegraphics[width=3in]{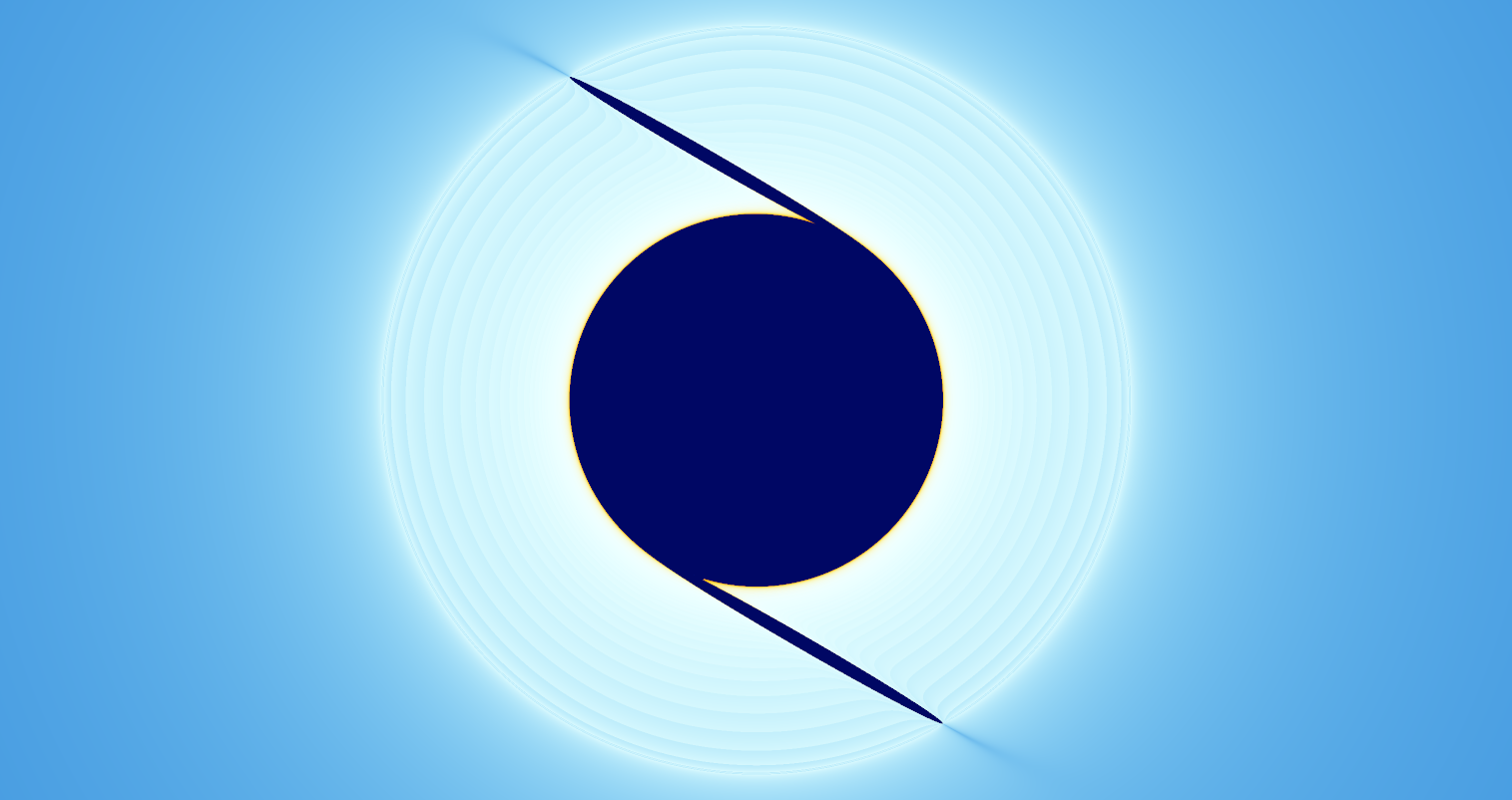}
    \caption{The slice in parameter space given by $s=-0.5+0.865i$.}
    \label{fig:cubic4}
\end{figure}

\section{Concluding Remarks}

There are several natural questions that are left open as a result of this work. Theorem \ref{thm:structure} shows that a cubic Blaschke product has three zeros, two critical points, and one hyperbolic inflection point. It seems reasonable to conjecture that if $B$ is a degree $d$ Blaschke product, then $B$ has $d-k$ solutions to $(\Delta_z^k B)(z)=0$ in $\D$. If this is true, then one could ask whether the unique solution to $(\Delta_z^{n-1} B)(z)= 0$ lies at the hyperbolic midpoint of the two solutions of $(\Delta_z^{n-2} B)(z) = 0$.

Due to the complicated nature of the expression $P(r,s)$ in Theorem \ref{thm:dynamics}, we were not able to characterize which parabolic parameters yield $J(B) = \partial \D$. We do know that in the $s=0$ slice, these parameters are given by the cusps of the nephroid \cite{Fle15}, and thus it seems plausible to conjecture that the cusps in other $s$-slices yield these parameters.

\end{document}